\documentclass{amsart}
\usepackage{graphics}
\usepackage{amsthm}
\usepackage{epsfig}
\usepackage{color}
\usepackage[pdftex]{hyperref} 
 
\newcommand{\norm}[1]{\left\Vert#1\right\Vert}

\newcommand{\set}[1]{\left\{#1\right\}}
\newcommand{\eps}{\varepsilon}
\newcommand{\grad} {\nabla}

\newcommand{\bdary} {\partial}

\newcommand{\dd} {\; \mathrm{d}}

\newcommand{\bq}{\begin{equation}}
\newcommand{\eq}{\end{equation}}

\newcommand{\x}{\mathbf{x}}
\newcommand{\uc}{{u^c}}
\newcommand{\R}{\mathbb{R}}
\newcommand{\e}{\epsilon}

\newcommand{\ac}{\mathbf{\theta}}

\newtheorem{thm}{Theorem}[section]
\newtheorem{cor}[thm]{Corollary}
\newtheorem{defn}[thm]{Definition}

\theoremstyle{remark}

\newtheorem{lemma}{Lemma}

\begin{document}

\title[The Dirichlet problem for the convex envelope]{The Dirichlet problem for the convex envelope}
\author{Adam~M. Oberman}
\address{Department of Mathematics, Simon Fraser University}
\email{aoberman@sfu.ca}
\author{Luis Silvestre}
\address{Department of Mathematics, University of Chicago}
\email{silvestr@cims.nyu.edu}

\date{\today} 	
\begin{abstract}
The Convex Envelope of a given function was recently characterized as the solution of a fully nonlinear Partial Differential Equation (PDE).   
In this article we study a modified problem: the Dirichlet problem for the underlying PDE. 
The main result is an optimal regularity result.  Differentiability ($C^{1,\alpha}$ regularity) of the boundary data implies the corresponding result  for the solution in the interior, despite the fact that the solution need not be continuous up to the boundary.    
Secondary results are the characterization of the convex envelope as:
(i) the value function of a stochastic control problem, 
and (ii) the optimal underestimator for a class of nonlinear elliptic PDEs.   
\end{abstract}   

\maketitle


\section{Introduction}
In this article we study a problem which is at the interface of Convex Analysis and nonlinear elliptic Partial Differential Equations.   The object of this study is the Convex Envelope.  In a previous work by one of the authors~\cite{ObermanCE}, a Partial Differential Equation (PDE) in the form of an obstacle problem for the convex envelope was obtained.    In this article we further explore the connection between the convex envelope and nonlinear elliptic PDEs by studying the Dirichlet Problem for the convex envelope.  

The convex envelope, $u$, of given boundary data, $g$, on a bounded domain, $\Omega \subset \R^n$, with boundary, $\partial \Omega$, 
is the solution of the following fully nonlinear and degenerate elliptic Partial Differential Equation,
\begin{align}\tag{PDE}\label{lam}
\lambda_1[u](x) = 0,   &\qquad \text{ for $x$ in $\Omega$},\\
u(x) = g(x), &\qquad \text{ for $x \in \partial\Omega$}.
\tag{D}\label{D}
\end{align}
Here $\lambda_1[u]$ denotes the smallest eigenvalue of the Hessian, $D^2u$, of the function $u$.

The solution describes a convex function which is nowhere strictly convex, in a sense which we explain later.   We refer to the solution of~\eqref{lam},\eqref{D} as the convex envelope of the boundary data $g(x)$, since these solutions agree with the usual definition~\eqref{CEdefn} (below).

While the convex envelope has been well-studied, it is useful to have a different characterization.
In fact, this characterization has led to PDE-based numerical methods for computing the convex envelope~\cite{ObermanCENumerics}.  

\subsubsection*{Continuity at the boundary}
Note that while we study the Dirichlet boundary condition~\eqref{D} for~\eqref{lam}, 
the solution need not achieve the boundary values (for example if the boundary data is nonconvex).  
In addition, in dimensions three or higher, the convex envelope itself need not be continuous up to the boundary~\cite{KruskalCounterExample}.

\subsubsection*{Interior regularity}
The first regularity result  establishes 
$C^{1,\alpha}$ regularity of the solution in the interior, provided the boundary data is $C^{1,\alpha}$.   

This result is unusual for the following reason:
loosely speaking, regularity results for elliptic PDEs (see~\cite{CaffCabreBook} or~\cite{LuisRegularity} for more references) often fall into one of two categories.
The first is regularity from the boundary.  For example, solutions of uniformly elliptic PDEs are $C^{1,\alpha}$ regularity up to the boundary, provided the boundary data is $C^{1,\alpha}$ (Theorem 9.31 of~\cite{GilbargTrudinger}).

 The second is interior regularity, assuming nothing on the boundary nevertheless forces some regularity in the interior.
Our first regularity result is of the first type.   However it has the unusual feature that despite the smooth boundary condition, the regularity holds only in the interior.

\subsubsection*{Geometry of the contact set}
The second result we obtain concerns the geometry of the contact set of the solution with a supporting hyperplane.   As a consequence of this result, it is established that at each point in the interior, there is at least one direction in which the solution is linear.   This result is used to show that viscosity solutions satisfy~\eqref{lam} in a classical sense even though they may not be differentiable.

\subsubsection*{Stochastic control interpretation}
We give an interpretation of the convex envelope as the value function of a stochastic control problem.   

\subsubsection*{Application to degenerate elliptic PDEs}
Pursuing the connection with elliptic Partial Differential Equations further we show that the convex envelope of Dirichlet data is a natural subsolution of a class of degenerate elliptic PDEs. 

Viscosity solutions of~\eqref{lam} are also viscosity solutions of the  elliptic Monge-Amp\`ere equation with zero right hand side.   Thus the results for~\eqref{lam} also apply for the  elliptic Monge-Amp\`ere equation with zero right hand side. 

The equation~\eqref{lam} is highly degenerate: it is degenerate in every direction except the direction of the eigenvector of the first eigenvalue of the Hessian.   
Another highly degenerate equation is the Infinity Laplace equation~\cite{CrandallTour}, which is degenerate in every direction except the direction of the gradient.

\subsection{The Convex Envelope}
The convex envelope of a given function $g(x)$ is mathematical object which has been the subject of study for many years. 
The convex envelope of the function $g(x)$ is defined as the supremum of all convex functions which are majorized by $g$,
\[
u(x) = \sup\{  v(x) \mid v \text{ convex, } v(y) \leq g(y) \text{ for all } y  \in \R^n \}.
\]
It was recently observed~\cite{ObermanCE} that the convex envelope is the solution of a partial differential equation. 

The equation for the convex envelope, $u$, of the function $g:\R^n \to \R$, is 
\begin{equation}\label{obstacle}
\max \left\{  u(x) - g(x), -\lambda_1[u](x)   \right\} = 0.
\end{equation}
Here $\lambda_1[u](x)$ is the smallest eigenvalue of the Hessian $D^2u(x)$.  
 See Figure~\ref{fig1}. 
\begin{figure}
\scalebox{.75}{
\includegraphics{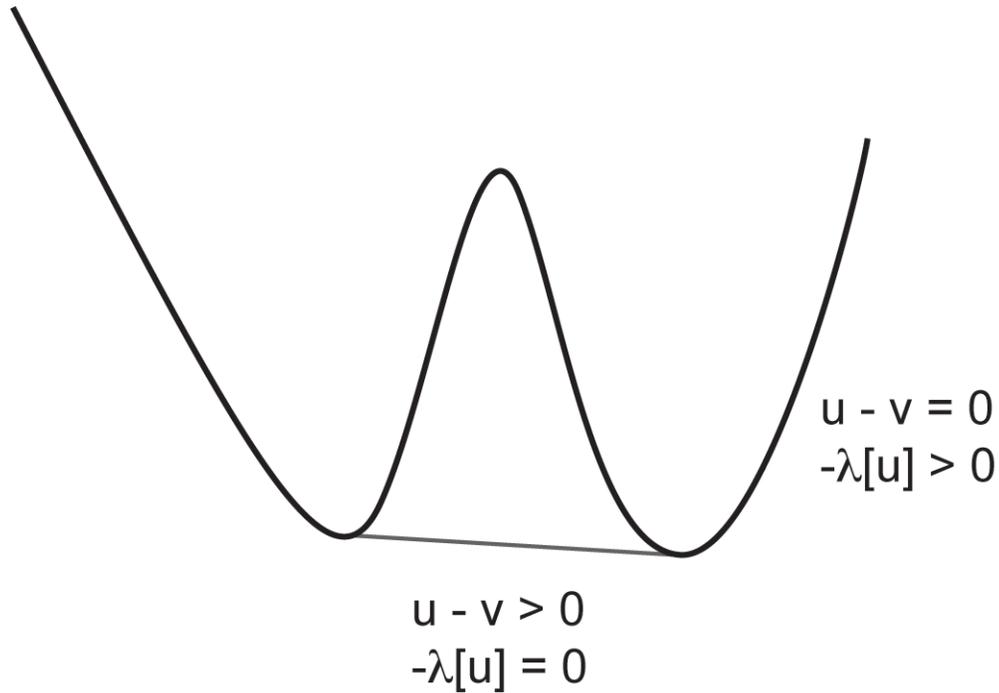}
}
\caption{Illustration of the equation}\label{fig1}
\end{figure}

The equation~\eqref{obstacle} is a combination of a fully nonlinear second order PDE,  $\lambda_1[u](x)$, and an obstacle term.   In dimension $n=1$, this equation coincides with  the classical obstacle problem.   But the differential operator is fully nonlinear in dimensions $n=2$ and higher and it is degenerate in all but one direction.

\subsection{Related results}
The possibility of an equation for the convex envelope was suggested by~\cite{GriewankRabier}.
A computational method for computing the convex envelope which used a related equation was perfomed in~\cite{vese}.  Methods for enforcing convexity constraints in variational problems have been studied as well~\cite{CarlierLRMaury}.  For more references on computational work, see~\cite{ObermanCENumerics}.

The regularity of the convex envelope $u(x)$ of a given function $g(x)$ has been studied by~\cite{KirchheimKristensen, GriewankRabier, Benoist}.   In this case, it has been shown that when the envelope function $g(x)$ is $C^{1,\alpha}$, the convex envelope is as well.
The analysis for the envelope problem is somewhat easier, since supporting hyperplanes will touch the function $g$ at some point where it is differentiable with matching derivatives.

In dimensions three or higher, the convex envelope of a function defined on a closed convex set need not be continuous up to the boundary~\cite{KruskalCounterExample}.

\section{A Partial Differential Equation for the Convex Envelope}
The convex envelope can be defined on  a domain  $\Omega$ in $\R^n$ even when $g(x)$ is defined only on part of the domain.   
This is achieved simply by setting $g$ to be infinity outside the domain of definition.

Although much of our arguments still hold for closed sets $C$, there is the possibility that the envelope is infinite in parts of the domain.  Since here our interest is in the relationship of the convex envelope with elliptic PDEs,  we restrict our attention to case where $C$ is $\partial \Omega$, the boundary of $\Omega$, and assume $g(x)$ is defined and bounded on $C$.

Then we define the convex envelope of the boundary data $g(x)$ to be
\begin{equation}\label{CEdefn}\tag{CE}
u(x) = \sup\{  v(x) \mid v \text{ convex, } v(y) \leq g(y) \text{ for all } y  \in \partial \Omega \}
\end{equation}
In this last definition, the set of convex functions $v$ can be replaced with the set of affine functions, since a convex function is the supremum of its supporting planes. 

The function defined in~\eqref{CEdefn} will not usually be continuous up to the boundary, since $g(x)$ may be concave at some points.  
However, we cannot simply take a convexification of $g$ and expect the same result, since $g$ is defined on lower dimensional set.

We now focus our attention on the Dirichlet problem~\eqref{lam},\eqref{D}.
The problem~\eqref{lam},\eqref{D} can be obtained from~\eqref{obstacle} by setting $g(x)$ to be infinite on the interior of the domain.

Analogous results for concave functions can be made by replacing $\lambda_1$ with $\lambda_n$, the largest eigenvalue.

\subsection{Convexity}
For basic definitions of convexity, see~\cite{BoydBook} or the appendix of~\cite{Evansbook}. 
For more on convex analysis see the textbooks ~\cite{BertsekasBook} or~\cite{Rockafellar}.
The function $u:\R^n\to \R$ is convex if for all $x,y\in R^n$ and  $0\le t \le 1$
\[
u( tx + (1-t) y) \leq t u(x) + (1-t) u(y).
\]
If $u:\R^n\to \R$ is convex, then for each $x_0 \in \R^n$ there exists a \emph{supporting hyperplane} to $u$ at $x_0$. 
In other words, there exists an 
affine function $P(x)\equiv u(x_0) + p\cdot (x-x_0)$ such that
\bq\label{supporting}\tag{SH}
\begin{cases}
u(x) \geq P(x)    &\text{ for all } x \in \R^n\\
u(x_0) = P(x_0) 	   &
\end{cases}
\eq
If $u$ is differentiable at $x$, then $\nabla u(x_0) = \nabla P(x_0)$ and $P$ is unique; if not, there may be more than one supporting hyperplane.

For twice-differentiable functions, convexity can be characterized by the local condition that the Hessian of the function be nonnegative definite, 
\[
D^2 u(x) \ge 0  \text{ for all } x \in \R^n   
\quad \text{ if and only if } \quad  u \text{ is convex. }
\]
Note that the first condition is equivalent to $\lambda_1[u](x) \ge 0$.  
The characterization is valid even for continuous functions, provided the equation is interpreted in the viscosity sense, as will be explained in the next section. 

\subsection{Viscosity Solutions}
The theory of viscosity solutions is a powerful tool for proving existence, uniqueness and stability results for fully nonlinear elliptic equations.  The standard reference is the User's Guide~\cite{CIL}.
A readable introduction is the Primer, by Crandall \cite{CrandallPrimer}.  
The theory applies to scalar equations of the form $F[u] \equiv F(D^2u(x), Du(x), u(x),x)$, which are nonlinear and elliptic, i.e. nondecreasing in the first argument.  
Solutions are \emph{stable} in the sense that if the equations $F^\e$ converge to $F$, the corresponding solutions $u^\e$ converge to the solution $u$ of $F$, uniformly on compact subsets.  Uniqueness is a consequence of the \emph{comparison principle}: if $F[u] \geq F[v]$ in $\Omega$, with $u\geq v$ on $\partial \Omega$, then $u\geq v$ in $\Omega.$
Viscosity solutions can be constructed using Perron's method, 
\[
u(x) = \sup \{ \phi(x) \mid  \phi \text{ is a subsolution of }F[\phi] = 0   \}.
\]
This last construction will coincide in our case with the definition of the convex envelope~\eqref{CEdefn}.
The definition of viscosity solutions for~\eqref{lam} follows.
 (In the interest of clarity, we omit a standard technical argument which requires working with upper and lower semi-continuous envelopes, see Section 4 in~\cite{CIL}.)
\begin{defn}\label{lambdadefn}
\label{defn:viscConvex}
The continuous function $u$ is a viscosity subsolution of \eqref{lam} if
\[
u(x) \leq g(x), \quad \text{ for } x \in \partial \Omega
\]
and if for every twice-differentiable function $\phi(x)$,  
\bq\label{lambdasub}
-\lambda_1[\phi](x) \leq 0, \quad \text{ when $x$ is a local max of $u-\phi$. }
\eq
The continuous function $u$ is a viscosity supersolution of \eqref{lam}  if 
\[
u(x) \ge g(x), \quad \text{ for } x \in \partial \Omega
\]
and if for every twice-differentiable function $\phi(x)$,  
\bq\label{lambdasup}
-\lambda_1[\phi](x) \geq 0, \quad \text{ when $x$ is a local min of $u-\phi$. }
\eq
The function $u$ is a viscosity solution of \eqref{lam} if it is both a subsolution and a supersolution.
\end{defn}

The following theorem was proved in~\cite{ObermanCENumerics}.

\begin{thm}\label{thm:char}
The  continuous function 
$u: \R^n \to \R$ is convex if and only if it is a viscosity subsolution of~\eqref{lam}.
\end{thm}

An immediate consequence of the previous theorem and Perron's method is the characterization of the convex envelope of the boundary data  as a viscosity solution of~\eqref{lam}.   This result was given in more detail in~\cite{ObermanCE} for the convex envelope of a function (rather than the Dirichlet data). 

Compare the Perron formula
\[u(x) = \sup 
\{ v(x) \mid v(x) \text{ is a {continuous} subsolution of~\eqref{lam} } \}.
\]
and the definition~\eqref{CEdefn}.  Use the definition of viscosity solutions and Theorem~\ref{thm:char} to see that the supremum is over the same set of functions.  (Here semi-continuity of the function $u(x)$ defined above follows from the definition.   Continuity can be established by an additional technical argument, see Theorem 4.1 in~\cite{CIL}.)

\subsection{Application to Monge-Amp\`ere equation with zero right hand side}
The Monge-Amp\`ere equation with zero right hand side is 
\[
\det(D^2 u) = 0
\]
In order for the equation to be elliptic, there is additional constraint that $u$ be convex.
While the next result is probably not new,  the characterization of the convex envelope~\eqref{lam} gives a clear and concise proof and statement.

\begin{lemma}
Viscosity solutions of the  elliptic Monge-Amp\`ere equation with zero right hand side
are also viscosity solutions of~\eqref{lam}
\end{lemma}

\begin{proof}
Using $\lambda_1[u] \geq 0$ to enforce the convexity constraint, and writing
\[
\lambda_1[u] \leq \dots \leq \lambda_n u[u]
\]
for the eigenvalues of the Hessian of $u$ and noting that 
\[
\det(D^2u) = \lambda_1[u] \cdot \dots \lambda_n[u]
\]
we see that 
\[
\det(D^2 u) = 0, \text{ and } u \text{ convex }
\]
is equivalent to the simpler condition
\[
\lambda_1[u] = 0.
\qedhere
\]
\end{proof}


\section{Geometry of the Contact Set}
The geometrical characterization of the convex envelope~\eqref{obstacle} 
suggests that solutions of~\eqref{lam}  are convex, but nowhere strictly convex.    By studying the geometry of solutions we can give a rigorous meaning to this statement.     

It is easy to build examples to show that solutions of \eqref{lam}  need not be differentiable; for example $u(x,y)  = |x|$ is a solution.  
Neither must they be continuous up to the boundary.  
For example, when  $\Omega$ is a square in the plane, centered at the origin
and $g(x,y) = x^2 - y^2$, the solution of \eqref{lam} is $u(x,y) = x^2 -1$.

For classical solutions, by which we mean twice continuously differentiable ($C^2$),  the principal curvatures of the graph of the solution must be $0, \lambda_2$, with $\lambda_2 \ge 0$.  This means that at every point in the domain, solutions are flat in the direction of the first eigenvector of the Hessian. 

However for nonclassical solutions, it is possible that this direction is changing.  The direction of flatness has implications for the regularity of solutions and for the stochastic control interpretation.  

Nonclassical example solutions have the property that at any point $x$ there is a direction $d = d(x)$ on which the solution restricted to the line $\ell(s) = x + ds$ is linear.
\begin{quote}
Is it true that for all points $x$ in the domain, there is a line segment containing $x$  on which the solution is \emph{linear}?
\end{quote}
To answer this question, we investigate the geometry of the contact set.


\begin{defn}[Contact Set]
Let $u$ be a solution of~\eqref{lam} and let $P^x$ be a supporting plane~\eqref{supporting} to $u$ at $x$.  Define the contact set to be the set of points where the supporting plane touches the graph of the function.
\bq\label{contact_set}
C^x = \{ y \in \Omega \mid  u(y) = P^x(y) \}.
\eq
\end{defn}
The contact set may also depend on the choice of the plane $P^x$, if there is more than one supporting hyperplane at $x$.

\subsection{Statement and Proof of the Theorem}

\begin{thm}\label{lem:CS}\label{thm:contact}
Let $\Omega$ be a compact domain in $\R^n$ and $g :\bdary \Omega \to \R$. Let $u : \Omega \to \R$ be its convex envelope. For any point $x$ inside $\Omega$, let $P^x$ be its supporting plane and $C^x$ its contact set.  Then
\begin{enumerate}
\item $C^x$ is convex.
\item $C^x$ intersects the boundary of $\Omega$.
\item \label{lem3}$C^x$ contains a line segment connecting $x$ to a point on the boundary of $\Omega$.
\item $C^x$ is the convex hull of $C^x \cap \bdary \Omega$.
  \end{enumerate}
 \end{thm}

\begin{proof}
(1) For any $y, z \in C^x$, let $w = (y+ z)/2$.   Since $u$ is convex, 
\[
u(w) \leq \frac{u(y) + u(z)} 2 = \frac { P^x(y) + P^x(z)} 2
\]
On the other hand, since $w \in C^x$, and $P^x$ is a supporting hyperplane, 
\[
u(w) \ge P(w) = \frac { P^x(y) + P^x(z)} 2.
\]
So $C^x$ is convex.

(2) Suppose $C^x$ does not intersect the boundary of $\Omega$.  Since $C^x$ is convex, it is closed, so there is a small neighborhood $D$ of $C^x$ which also does not intersect the boundary of $\Omega$.   Since $P^x$ is a supporting hyperplane, $u > P^x$ in $D\setminus C^x$.  Let 
\[
\e = \min_{ y \in \partial D } u(y) - P^x(y).
\]
Then $\e > 0$.  So define $\tilde u(y) = \max ( u(y), P^x(y) + \e/2 )$.  Then $\tilde u(y) > u(y)$ at some point $y$ in $D\setminus C^x$, but $\tilde u(y) = u(y)$ outside $D$.  This contradicts the definition of $u$,~\eqref{CEdefn}.

(3) Suppose $C^x$ intersects the boundary of $\Omega$ at $y$.  Then since $C^x$ is convex, there must be a line segment from $x$ to $y$.

(4) Let $D$ be the convex hull of $C^x \cap \bdary \Omega$. Since $C^x$ is closed, so is $D$. Moreover, since $C^x$ is convex, $D \subset C^x$. We are left to prove the opposite inclusion.

Let $y \in \Omega \setminus D$. We will show that $y \notin C^x$. Since $C^x$ is closed and convex, then by the Plane Separation Theorem~\cite{BertsekasBook} there is an affine function $L$ such that $L(y)>0$ and $L<0$ in $D$.
 
Since $L<0$ in $D$, then $g-P^x$ is strictly positive in the compact set $\bdary \Omega \cap \{L \geq 0\}$. Therefore, there is a $\delta>0$ such that $g-P^x>\delta$ in $\bdary \Omega \cap \{L \geq 0\}$. Thus, for $\eps>0$ small enough, $P^x + \eps L < g$ on $\bdary \Omega$.

Since $u$ is above any affine function that is below $g$ on $\bdary \Omega$, then $u(y) \geq P^x(y) + \eps L(y) > P^x(y)$. So $y \notin C^x$, which finishes the proof.
\qedhere

\end{proof}

Let $x \in C^x$, then by the last part of the theorem, 
\bq\label{xk}
x \in \text{conv}[y_1,\dots, y_k], \quad \text{ for } y_i \in \partial \Omega
\eq
for  $i = 1,\dots, k$ with $2 \le k \le n$
(where $\text{conv}[y_1,\dots, y_k]$ means  the convex hull of  $y_i$).
This means that $u$ is affine on a $(k-1)$ dimensional set, which contains $x$ in its relative interior.

\subsection{Consequences of the Geometry}
Theorem~\ref{thm:contact} means that, even if the solution is not differentiable, it always satisfies~\eqref{lam} in an (almost) classical sense.  Recall the definition of the second directional derivative.

\begin{defn}[Second directional derivative]
Given a unit vector $|v| = 1$, the second derivative of $u$ in the direction $v$ is given by
\[
\frac{d^2 u}{dv^2} \equiv   \lim_{h,k \to 0^+} 
\frac{ h}{h+k} \left (  \frac{ u(x+vk) - u(x) }{k} \right) 
+
\frac{ k}{h+k} \left (  \frac{ u(x) - u(x-vh) }{h} \right) 
\]
\end{defn}

\begin{cor}
Let $u(x)$ be the convex envelope of the boundary data, defined on a domain $\Omega$.  
Then 
\[
\lambda_1[u](x) = \min_{|v| = 1} \frac{d^2 u}{dv^2}(x) = 0
\]
for all points $x$ inside $\Omega$.
\end{cor}
\begin{proof}
For all $h, k$, using the definition of a second directional derivative as a limit, followed by the minimum characterization of the smallest eigenvalue of the Hessian, gives:
\[
\lambda_1[u](x) = \min_{|d| = 1}  \lim_{h,k \to 0^+} 
\frac{ h}{h+k} \left (  \frac{ u(x+vk) - u(x) }{k} \right) 
+
\frac{ k}{h+k} \left (  \frac{ u(x) - u(x-vh) }{h} \right) 
\]
Since $C^x$ is the convex hull of $C^x\cap \partial \Omega$, $u(x)$ is linear on a $k-1$ dimensional affine set containing $x$, which is given by~\eqref{xk}.
So choosing any direction $d$ which lies in this set, we get zero in the last equation.
\end{proof}

\section{Stochastic Control Interpretation}
In~\cite{ObermanCE} the convex envelope was reinterpreted as the value function of a stochastic control problem. For the Dirichlet problem~\eqref{lam}, an even simpler interpretation is available.
Our derivation is formal but can be made rigorous; we refer to~\cite{FlemingSonerBook} for a rigorous derivation of related equations.   
For readers not familiar with stochastic control problems, an introduction to optimal control and viscosity solutions of Hamilton-Jacobi equations can be found in~\cite{Evansbook}.

\subsection{The Control Interpretation}
Consider the controlled diffusion
\begin{equation}\label{diffusion2}
\begin{cases}
d\mathbf x(t) =  \sqrt{2}\, \ac(t)  dw(t),\\
\phantom{d}\x(0) =  x_{0}.
\end{cases}
\end{equation}
where $w$ is  a one-dimensional Brownian motion, and the control, $\ac(\cdot)$, is a mapping into unit vectors in $\R^n$.   The process stops when it reaches the boundary of the domain, at which point we incur a cost $g(x_\tau)$, where $\tau$ is the time when it reaches the boundary.
The objective is to minimize the expected cost
\[
 J(x,\ac(\cdot)) \equiv E^x[g(x_\tau)]
\]
over the choice of control $\ac(\cdot)$.
The value function is 
\[
u(x) = \min_{\ac(\cdot)} J(x,\ac(\cdot)),
\]
which describes the minimal expected cost at a given initial point $x$, assuming that an optimal control strategy was pursued.

In the general setting of stochastic control problems, the value function satisfies a fully nonlinear Partial Differential Equations, which is obtained by applying the Dynamic Programming Principle (DPP).
The DPP provides a link between nearby values of the value function by assuming a constant control pursued over an infinitesimal time interval.   A readable introduction to this principle in the deterministic case can be found in~\cite{Evansbook}.

Now apply the DPP to derive the equation for the value function.
One strategy is to fix $\ac(\cdot) = \theta$ to be constant,  and let the diffusion proceed for time $t$, thereafter following the optimal strategy.   This strategy costs $E^{x_0,\theta}[u(\x(s))] \equiv E^{x_0}[u(\x(s)) \mid \ac(\cdot) = \theta]$.
Minimizing over $\theta$ gives
\[
u(x_{0}) =  \min_{\theta} E^{x_0, \theta}[u(\x(t))]  + o(t).
\]
Using the definition of infinitesimal generator corresponding to the diffusion~\eqref{diffusion2}, with $\theta$ fixed (see e.g.~\cite{Oksendal}), gives
\begin{equation}\label{generatorL1}
\lim_{t\to 0} \frac{ E^{x,\theta}[u(\x(t))] - u(x) }{t} = \ac^T\, D^2 u(x)\, \ac.
\end{equation}
Using~\eqref{generatorL1} in the preceding equation and taking the limit $t\to 0$ yields 
\[
 -\inf_{|\ac| = 1}  \ac^T \, D^2 u(x)\, \ac = 0,
\]
along with $u= g$ on the boundary.
Finally, using the Raleigh-Ritz characterization of the eigenvalues 
\[
\lambda_1(M) = \min_{|x| = 1} x^T M x,
\]
for symmetric matrices, $M$, 
recovers~\eqref{lam}.

\newcommand{\dist} {\mathrm{dist}}

\section{$C^{1,\alpha}$ Regularity}

In this section we prove a regularity result.
We begin with an example to show that the interior regularity result is optimal in the sense that it cannot be continued up to the boundary.

We will use the notation
\begin{align*}
B_r = \{ x \in \R^n \mid |x|^2 < r \},
&&
S_r = \{ x \in \R^n \mid |x|^2 = r \}.
\end{align*}
for the ball of radius $r$, and the sphere of radius $r$, respectively.

\subsection{Optimal interior regularity example}\label{ex:interiorreg}
Let $\Omega = B_1 \subset \R^2$, the unit ball in two dimensions.
Consider the function 
\[
u(x,y) = - (1+x)^{1-\e}
\]
for $\e>0$ small.  The function is convex and continuous on $B_1$.
Writing 
\[
u(x,y) = f(r, \theta) = -(1 + r\cos\theta)^{1-\e}
\]
in polar coordinates, we see that $u$ restricted to the unit circle is $C^{1, 1-2\e}$,  since near $\theta = \pm \pi$ the function behaves like  $\theta^{2 - 2\e}$.
Compute the derivative, 
\[
\frac{\partial}{ \partial \theta} f(1,\theta) = (1-\e)\sin\theta (1 + \cos\theta)^{-\e} 
\]
which has a singularity about $\theta = \pm \pi$ which is on the order of 
$\theta^{1-2\e}$.  

Notice that the function $\frac {\partial}{ \partial \x}u(x,y)$ has a singularity at $(-1,0)$.  So the function $u(x,y)$ is not $C^{1,\alpha}$ up to the boundary of $B_1$ for any $\alpha>0$.

Thus we have provided the desired counterexample to the regularity of the solution of~\eqref{lam},\eqref{D} near the boundary.

\subsection{Proof of the Regularity result}
Let $\Omega = B_1 \subset \R^n$ and let  
$g : S_1 \subset \R^n \to \R$ be a  $C^{1,\alpha}$ function.
Consider~\eqref{lam},\eqref{D} 
\[
\begin{cases}
\lambda_1[u](x) = 0, & \text{ for $x$ in } B_1
\\
u  = g, &\text{ for $x$ on } \bdary B_1 = S_1 
\end{cases}
\]
This equation is understood in the viscosity sense.  
Reinterpreting the definition from the Perron construction gives the characterization
\[
\begin{cases}
u(x)= \max \set{L(x) : L(x) = A \cdot x + b , L \leq g \text{ on } \bdary B_1 }, &  \text{ for $x$ in } B_1
\\
u = g  &\text{ for $x$ on } \bdary B_1 = S_1 
\end{cases}
\]

We recall that if the boundary data $g$ is the restriction of a nonconvex function, then the solutions will not be continuous up to the boundary.

We next prove the interior regularity result.

\begin{thm} 
Consider~\eqref{lam},\eqref{D}, for $\Omega = B_1 \subset \R^n$.
If the boundary data $g$ is $C^{1,\alpha}$ on $\partial\Omega = S_1$ for some $\alpha \in (0,1]$, then the convex envelope $u$ is $C^{1,\alpha}$ in $B_{1/2}$. 
Moreover, the following estimate holds, 
\[ 
\norm{u}_{C^{1,\alpha}(B_{1/2})} \leq C \norm{g}_{C^{1,\alpha}(\bdary B_1)} 
\]
where the constant $C$ depends only on the dimension $n$ and on $\alpha$.
\end{thm}

\begin{proof}
~\,~
\subsubsection*{Preliminaries, Rescaling}
We can assume 
\[
\norm{g}_{C^{1,\alpha}(\bdary B_1)} = 1
\]
because if we can prove the estimate in this case, the general case will follow by rescaling. 

To prove the estimate, we will choose two points, 
 $x_1, x_2 \in B_{1/2}$,  and establish the estimate 
\[ 
|\grad u(x_1) - \grad u(x_2)| \leq C |x_1-x_2|^\alpha
\]
for some constant $C$ depending only on the dimension $n$ and $\alpha$.
First replace the gradients with the gradients of supporting hyperplanes.
So let
\bq\label{sh}
L_i = A_i \cdot (x- x_i) + b_i, \quad i=1,2 
\eq
be the supporting hyperplanes to $x_i$, respectively.
The argument will show that the supporting hyperplanes are unique.
Define 
\[
M(x) = \max(L_1(x),L_2(x)).
\]
See~\autoref{f:points}(a).

\begin{figure}[hbt] \begin{center}
\hspace{-.5in}\scalebox{.4}{\includegraphics{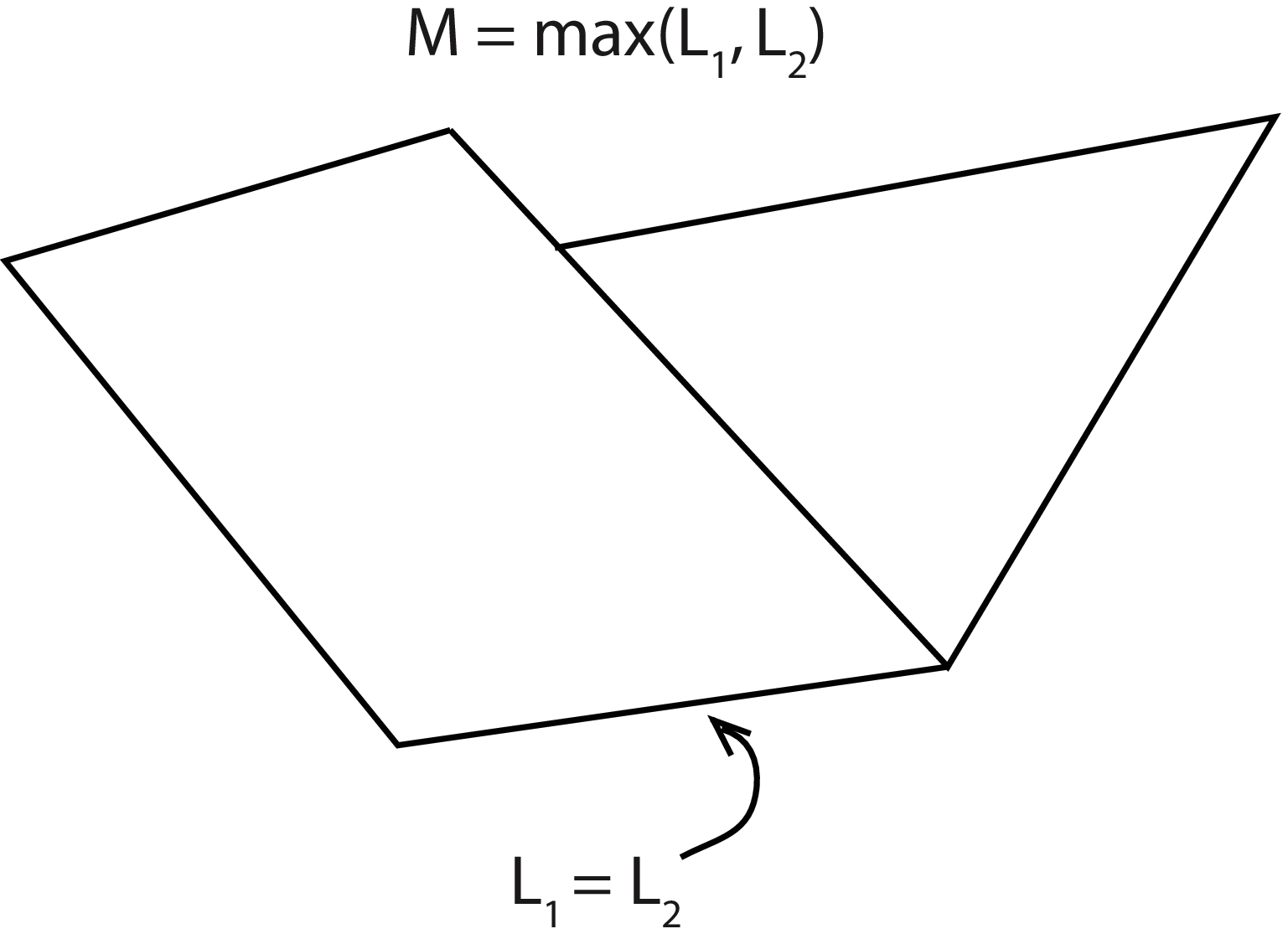}}
\hspace{-.4in}\scalebox{.375}{\includegraphics{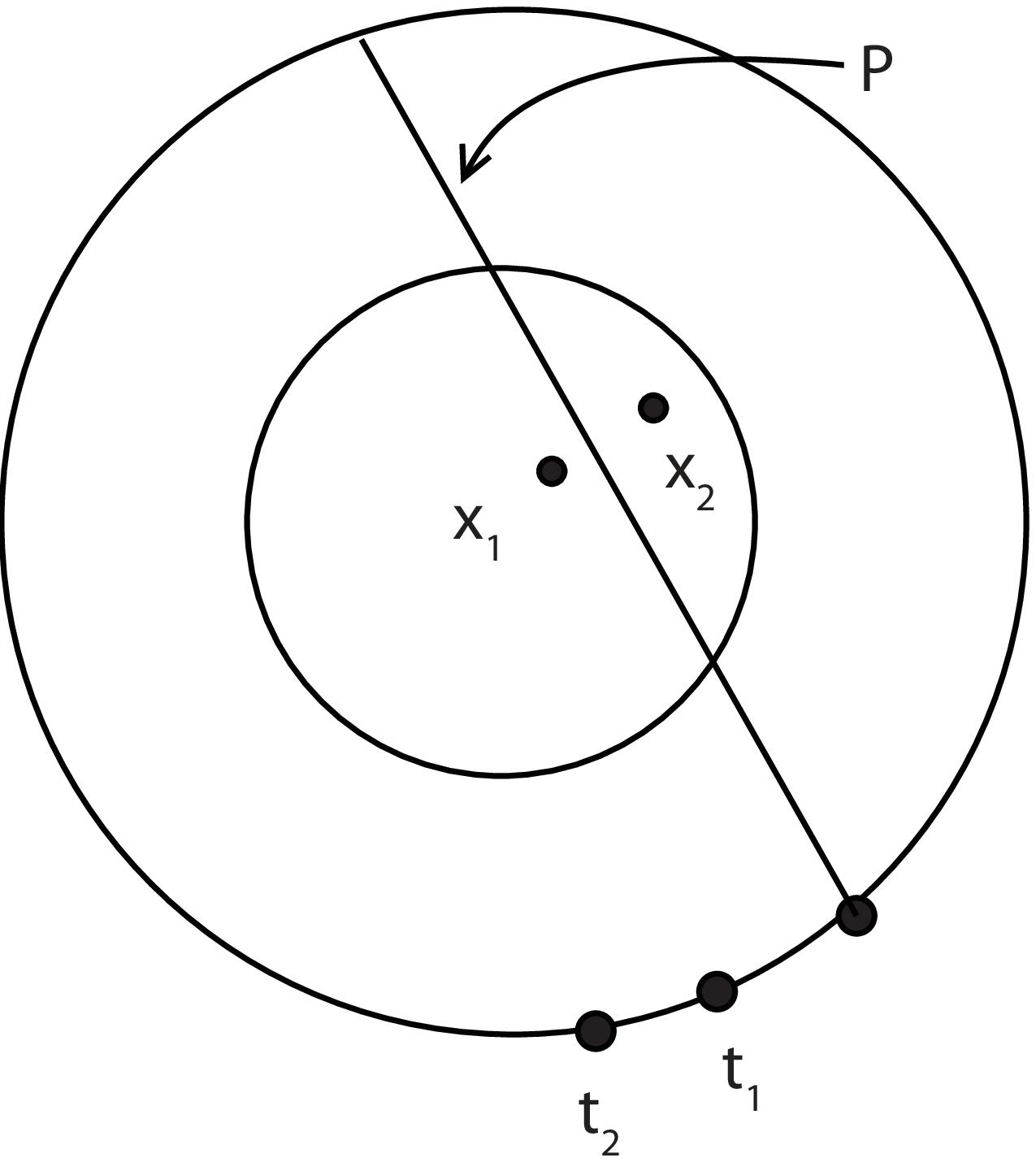}}
\end{center}
\caption{(a) The function $M = \max(L_1,L_2)$ (b)  $P$ separates  $x_1, x_2$ and meets $S_1$ at a bounded angle. }
\label{f:points}
\end{figure}

By subtracting an appropriate affine function from $u$ and $g$ we can assume without loss of generality that 
\begin{align*}
A_1+A_2=0
\end{align*}
and
\bq\label{Mzero}
\min_{x\in B_1} M=0. 
\eq
so the $(n-1)$ dimensional hyperplane $P = \{x \mid L_1(x) = L_2(x) \} $ can be written as 
\[
P= \{x \mid M(x) =0\}
\]
and note for later that we can write 
\bq\label{Meqn}
 M(x) = \frac{|A_1 - A_2|}{2} \dist(x,P)
\eq

As a result of these simplifications, our goal is to  show that 
\bq\label{goal1}
|A_1-A_2| \leq C |x_1-x_2|^\alpha
\eq
for some constant $C$ depending only on the dimension $n$ and $\alpha$.

The proof will proceed based on an estimate on $\min g$ from below in terms of $|A_1 - A_2|$, which will lead to an estimate on $|x_1-x_2|$ from below.

\subsubsection*{Comparing Euclidean and geodesic distances}
We will be comparing distances from points on the boundary  $S_1 = \partial B_1$ to the intersection of $P$ with $S_1$, both in the ambient space $\R^n$ and along the boundary, using the geodesic distance.   With this in mind, write $\dist(y, x)$ for the standard Euclidean distance, and 
write for $y \in S_1$,
\[
\dist_{S_1}(y, P \cap S_1)
\]
for the geodesic distance between $y$ and $P \cap S_1$ on $S_1$.   

Since $P$ separates $x_1$ and $x_2$,  $P$ passes through $B_{1/2}$ and intersects $S_1$ nontangentially  at some minimum angle $\theta_0 > 0$ (which depends only on the dimension $n$).  Refer to~\autoref{f:points}(b).
Thus for any point $y \in S_1$, 
the distance between $y$ and $P$ is greater than, but comparable with the geodesic distance between $y$ and $P \cap \bdary B_1$ on $\bdary B_1$.   
We record this assertion as follows. 
There is a constant $C_2 > 1$ such that for any $y\in S_1$
\bq\label{GeoDist}
 \dist(y,P) \leq  \dist_{S_1}(y, P \cap S_1) \leq C_2 \,\dist(y,P) 
\eq
where
\bq\label{c1c2}
 1 \leq C_2
\eq

\subsubsection*{Setup estimates using distances on the boundary}
Choose  $x_0 \in S_1$ so that  the minimum of $g$ is achieved at $x_0$,
\bq\label{g0}
\min_{x\in S_1} g(x) = g(x_0) \ge 0
\eq
and moreover, this minimum is nonnegative, by~\eqref{Mzero}.

Now select a great circle on $S_1$ which is perpendicular to $P$ and which 
passes through $x_0$.
Let $\gamma(t)$ be a parameterization by arclength of the great circle, which passes through $P$ and $x_0$
\[ 
\gamma(0) \in P, \qquad \gamma(t_1)=x_0.
\]
Using~\eqref{GeoDist}, noting the parameterization is by arclength, we obtain
\bq \label{ArcDist}
 \dist(\gamma(t),P) \leq t \leq C_2 \dist(\gamma(t),P),\quad  \text { for $0 < t < \pi/2$}
 \eq
where we have chosen $\gamma(0)$ in the same hemisphere as $x_0$.

By convexity of $u$ and the supporting hyperplane conditions, 
\[
g(x) \geq u(x)  \geq M(x),\quad   \text{ for $x$ in }\overline B_1. 
\]
For future reference, define the function of one variable $G(t) = g(\gamma(t))$ and record this last result as
\bq\label{gM}
G(t) \geq M(\gamma(t)),\quad  \text { for $0 < t < \pi/2$}
\eq

\subsubsection*{Apply estimates at two points}
We will proceed to apply~\eqref{ArcDist} along with~\eqref{gM} at $t_1$ and at a second, larger value of $t$ to obtain the desired result.

Combining~\eqref{gM} with the expression~\eqref{Meqn} for $M(x)$, 
we obtain
\begin{equation} \label{a1}
\dist(x_0,P) \le \frac{ 2 g(x_0) }{ |A_1 - A_2| }
\end{equation}
Applying~\eqref{ArcDist} at $t_1$ we obtain
\bq\label{a2}
t_1 \leq C_2 \dist(x_0,P).
\eq
Combining~\eqref{a2} and~\eqref{a1} we obtain
\bq\label{t1Est}
0 \le t_1  \leq C_2 \frac{ 2 g(x_0) }{ |A_1 - A_2| }. 
 \eq

Next let 
\bq\label{t2defn}
t_2 ={2C_2} \frac{ 2 g(x_0) }{ |A_1 - A_2| }
\eq
then 
\bq\label{t1t2}
  C_2 \frac{ 2 g(x_0) }{ |A_1 - A_2| } \le t_2 - t_1   \le   2C_2 \frac{ 2 g(x_0) }{ |A_1 - A_2| }
\eq

Then compute
\[ \begin{aligned}
G(t_2) &\geq  M(\gamma(t_2)) 
&& \text{ from~\eqref{gM} }\\
& = \frac{|A_1 - A_2|}{2}\dist(\gamma(t_2), P)  
&& \text{ using~\eqref{Meqn} } \\
& \geq  \frac{|A_1 - A_2|}{2} \frac{1}{C_2} t_2
&& \text{ by~\eqref{ArcDist}} \\
& = \frac{|A_1 - A_2|}{2} \frac{1}{C_2}{2C_2} \frac{ 2 g(x_0) }{ |A_1 - A_2| }
&& \text{ inserting~\eqref{t2defn}} \\
& = 2g(x_0) 
&& \text{ after simplifying }
\end{aligned}
\]
which we record as 
\bq\label{Gt2val}
G(t_2)  \geq  2 g(x_0)
\eq

\subsubsection*{Conclusions from the estimates}
Using ~\eqref{Gt2val} and the definition $G(t_1) = g(x_0)$ we obtain
\[
\int_{t_1}^{t_2} G'(t) \dd t = G(t_2) - G(t_1)  \geq  g(x_0) 
\]
Next we use the fact that
$
\int_{t_1}^{t_2} f(t)\dd t \geq \beta
$
implies there exists $t^*$ in $[t_1, t_2]$ such that $f(t^*) \geq \beta (t_2 - t_1)^{-1}$.
Apply this fact to the previous integral to conclude the following.
There exists  $t^*$ in $[t_1, t_2]$ such that
\bq\label{Gpt}
G'(t^*) \geq  g(x_0) (t_2 - t_1)^{-1} \ge \frac{1}{2C_2} \frac{|A_1 - A_2|}{2}
\eq
where we have again used~\eqref{t1t2}.

\subsubsection*{Relate to assumptions on $g$}
Since $g(x_0)$ is the minimum of $g$, 
\[
G'(t_1)=0.
\]
Moreover, since $\norm{g}_{C^{1,\alpha}(\bdary B_1)}=1$ then 
\[
|G'(s_1) - G'(s_2)|\leq |s_1-s_2|^\alpha, \quad \text{ for all } s_1, s_2
\]
In particular,
\[
G'(t^*) \le |t^*-t_1|^\alpha
\]
so substituting~\eqref{Gpt} into the last equation
\[
 \frac{1}{2C_2} \frac{|A_1 - A_2|}{2} 
 \leq |t^*- t_1 |^\alpha  \leq |t_2 - t_1|^\alpha 
 \leq   \left ( 2C_2 \frac{ 2 g(x_0) }{ |A_1 - A_2| } \right )^\alpha
\]
using~\eqref{t1t2}.
This last equation simplifies to
$
g(x_0) \ge (4C_2)^{-1 +1/\alpha}  |A_1-A_2|^{1+1/\alpha}
$
Thus, there is a constant $c$ such that 
\[
\min g = g(x_0) \geq c |A_1-A_2|^{1+1/\alpha}
\]
\subsubsection*{Conclude}
Therefore the constant function $L(x) = c |A_1-A_2|^{1+1/\alpha}$ is below $g$ on 
$S_1$, and then, by comparison
\[
u(x) \geq c |A_1-A_2|^{1+1/\alpha} \text{ for $x$  in $B_1$}
\]
Returning to the original points $x_1$ and $x_2$, and recalling the supporting hyperplanes~\eqref{sh}, we have
 \[
u(x_i) = M(x_i) = \frac{|A_1-A_2|}{2} \dist(x_i,P), \quad  \text{for $i=1,2$}
 \]
then 
\[
\dist(x_i,P) \geq c |A_1 - A_2|^{1/\alpha},
\]
 which clearly implies~\eqref{goal1} which was the desired result.
\end{proof}
 

\newcommand{\ellm}{\gamma}
\newcommand{\ellM}{\Gamma}
\newcommand{\Pmax}{\mathcal{M}^+_{\ellm, \ellM}}
\newcommand{\Pmin}{\mathcal{M}^-_{\ellm, \ellM}} 
 
\section{Estimators for elliptic PDEs.}
\subsection{Sharp underestimators for elliptic PDEs.}
Convex functions are natural candidate subsolutions for nonlinear elliptic partial differential equations.  The equation \eqref{lam} allows us to prove that the convex envelope is the sharpest subsolution for a class of equations, in the sense we describe below.  

Similar results about best \emph{overestimators} can be obtained for the \emph{concave} envelope of the boundary data, which is the solution of $-\lambda_n[u] = 0$.
\newcommand{\ub}{\underline{u}}

What we describe below is closely related to the maximal and minimal Pucci operators, which are the best possible sub- and super-solutions for a class of uniformly elliptic equations with given ellipticity constants, as in~\cite{CaffCabreBook}.  
We begin by showing that the convex and concave envelopes provide the best possible sub- and super-solutions for a class of possibly degenerate elliptic equations. 
We then review the uniformly elliptic case, following~\cite{CaffCabreBook}.
Finally  we show that the Pucci  maximal and minimal operators converge to the convex and concave envelope as the ratio of constants goes to zero or infinity, respectively.

To that end, consider the Dirichlet problem for a nonlinear elliptic partial differential equation,
$F[u](x) \equiv F(D^2u(x), Du(x), u(x),x)$.
For motivation, suppose that we have very little detailed information about the operator $F$.
Assume for clarity that $F$ satisfies
\bq\label{Fhomog}
F(0, p, r, x) = 0, \quad \text{ for all } p\in \R^n, r\in \R ,x \in \Omega.
\eq
Also assume that we have (precisely known or estimated) Dirichlet boundary conditions,
\[
u = g, \quad \text{ on } \partial \Omega.
\]
Then we show below that we can estimate the solution by the convex and concave envelope of the boundary data.   

The case~\eqref{Fhomog} can be generalized to $F(0,p,r,x) = f$, with  corresponding results hold for the operators $\lambda_1[u] = f$.

\subsection{Estimators without ellipticity bounds - envelope operators}
Without assuming that the operator is uniformly elliptic, we obtain the  equation for the best underestimator.  As expected, we get the convex envelope operator.

Define the \emph{best underestimator for elliptic equations} 
\bq\label{BestEst}
\ub = \inf_{F}  \{ u(x) \mid   F[u] = 0 \text{ in } \Omega, u = g \text{ on  } \partial \Omega, F \text{ satisfies } \eqref{Fhomog} \}.
\eq
The boundary conditions hold in the viscosity sense, which is why the result may not achieve the boundary conditions at all points.

\begin{thm}
Consider the class of nonlinear elliptic equations which are homogeneous in the sense of~\eqref{Fhomog} and satisfy the boundary data $g$.  
The best underestimator for this class is the convex envelope of the boundary data.
\end{thm}
\begin{proof}
We make note of the fact that the convex envelope of the boundary data, which we write as $\uc$,
is the solution of \eqref{lam} with boundary data $g$.  
  
Now simply apply the definition \eqref{BestEst} with $F = -\lambda_1$ to obtain $\ub \leq \uc$.
 
To obtain the other inequality, first verify that the convex envelope is viscosity subsolution of the equation.
We do this by checking Definition~\ref{lambdadefn}.  Suppose $x$ is a local max of $\uc- \phi$, for some $C^2$ function $\phi$.  
By Definition~\ref{lambdadefn}, $\lambda_1[\phi](x) \ge 0$, so $D^2\phi(x) \geq 0$.
 Now compute 
\[
F[\phi](x) \equiv F(D^2\phi(x), D\phi(x), \phi(x), x) \geq F(0, D\phi(x), \phi(x), x) = 0,
\] 
which follows since $F$ is elliptic, and by~\eqref{Fhomog}.  
Now since $\uc$ is the convex envelope of the boundary data,  $\uc \leq g$ on $\partial \Omega$.
Together these facts imply that $\uc$ is a subsolution of $F$.  The comparison theorem applied to $F$ yields $\uc \leq \ub$.

Together, these results imply $\ub = \uc$.
\end{proof}

\subsection{Estimators with ellipticity bounds - Pucci operators}
We next state the case where we have some additional information on the operator $F$.
Suppose in addition that 
\bq\label{elliptic}
\text{$F[u]$ is uniformly elliptic, with constants $0 < \ellm \leq n\ellM$.}
\eq 
This section follows~\cite{CaffCabreBook}, which should be referenced for the definition of uniform ellipticity in the fully nonlinear case, and for additional details.  
(For the fully nonlinear case where  $F$ is differentiable, uniform ellipticity is equivalent to the fact that the linearization of the nonlinear operator $F$ at any particular values should have eigenvalues bounded above and below by the ellipticity constants.)

Define the \emph{best underestimator for uniformly elliptic equations},
\[
\ub_{\ellm , \ellM}  =  \inf_{F}  \{ u(x) \mid  F[u] = 0 \text{ in } \Omega, u = g \text{ on  } \partial \Omega, F \text{ satisfies } \eqref{Fhomog} \text{ and } \eqref{elliptic} \}.
\]
Then there is an explicit equation for $\ub_{\ellm , \ellM}$.  It  is the solution of the Pucci minimal equation, \cite[sec.\ 2.2]{CaffCabreBook}, which can be written as
\bq\label{Pucci}
\Pmin[u] = - \left (\ellm \sum_{\lambda_i > 0} \lambda_i + \ellM \sum_{\lambda_i < 0} \lambda_i \right ),  \quad  \{\lambda_i\}_{i=1}^n \text{ eigenvalues of } D^2u.
\eq
In the two dimensional case, the result is the operator
\[
\Pmin[u] = - \left (\ellm  \lambda_1 + \ellM \lambda_2 \right ),
\quad  \lambda_1 \leq \lambda_2 \text{ eigenvalues of } D^2u, \quad \Omega \subset \R^2.
\]
This follows from the simple observation that in order for the $\Pmin[u] = 0$, the eigenvalues must have different signs.

There is a related equation for the best overestimator, which is the Pucci maximal equation, see~ \cite[sec.\ 2.2]{CaffCabreBook}.

\subsection{Convergence of Pucci Operators to the convex envelope}
We are now equipped to prove the following.
\begin{thm}
The best underestimator for uniformly elliptic equations converges uniformly on compact subsets to the best underestimator for elliptic equations, which is the convex envelope of the boundary data, as $\ellM/\ellm \to \infty$.
\end{thm}
\begin{proof}
The result follows from stability of viscosity solutions, once we show that $\Pmin[u] \to -\lambda_1[u]$ as $\ellM/\ellm \to \infty$.  
Divide through by $\ellM$ in \eqref{Pucci}, and take the limit $\ellM/\ellm \to \infty$, to give
$
\sum_{\lambda_i < 0} \lambda_i =0.
$
This last identity is equivalent to $\min_i \lambda_i = 0$, which recovers $\lambda_1 = 0$.
\end{proof}

\bibliographystyle{amsplain}
\bibliography{convexenvelope}

\end{document}